\documentclass[a4paper,12pt,reqno]{amsart}

\title{On ruled surfaces with big anti-canonical divisor and numerically trivial divisors on weak log Fano surfaces}

\author{Rikito Ohta}
\address{Department of Mathematics,
Graduate School of Science,
Osaka University,
Machikaneyama 1-1,
Toyonaka,
Osaka,
560-0043,
Japan.
}
\email{usamarotokame@gmail.com}

\author{Shinnosuke Okawa}
\address{
Department of Mathematics,
Graduate School of Science,
Osaka University,
Machikaneyama 1-1,
Toyonaka,
Osaka,
560-0043,
Japan.
}
\email{okawa@math.sci.osaka-u.ac.jp}

\date{\today}

\usepackage{mymacros,fullpage,autobreak}
\usepackage{amscd}
\usepackage[all]{xy}

\newcommand{\chara}{\operatorname{char}}
\newcommand{\norm}{\operatorname{norm}}

\begin{document}
\maketitle

\begin{abstract}
We investigate the structure of geometrically ruled surfaces whose anti-canonical class
is big. As an application we show that the Picard group of a normal projective surface whose anti-canonical class is nef and big is a free abelian group of finite rank.
\end{abstract}

\tableofcontents

\section{Introduction}
Let $X$ be a smooth projective surface over a field $\bfk$ whose anti-canonical divisor $-K_X$ is big. Although the classification of such surfaces is completely understood if $- K _{ X }$ is also nef, the question becomes much harder if we drop the nefness from the assumption. If $X$ is rational, then it is known to be a Mori dream space as is shown in \cite[Theorem 1]{MR2824848}. They also gave a kind of structure theorem for such surfaces \cite[Theorem 2]{MR2824848}.

Let us consider the case when $X$ is not rational.
Then $X$ is obtained by repeatedly blowing up a geometrically ruled surface (over a curve of positive genus) with big anti-canonical bundle,
since the Kodaira dimension of $X$ is $-\infty$.
In the first part of this paper, we study the structure of geometrically ruled surfaces whose anti-canonical divisor is big.
We will show the following theorem.

\begin{theorem}[$=$\pref{th:main_theorem}]\label{th:int1}
Let
$
 C
$
be a smooth projective curve of genus $g \ge 1$ and
$
 E
$
be an unstable vector bundle of rank $2$ on $C$.
Then $-K _{ \bP _{ C } (E) }$ is big if and only if
$
 E \simeq L \oplus M
$
for some line bundles $L$ and $M$ on $C$ such that
$
 \deg L- \deg M > 2 g - 2
$.
\end{theorem}

In \pref{pr:main_theorem_split_case} below, \pref{th:int1} is proven under the assumption that $E$ is decomposable.
In order to reduce the general case to this, we take a degeneration of $E$ to the direct sum of its Harder-Narasimhan factors and then apply \pref{pr:main_theorem_split_case} to it.
To conclude the proof, we \emph{a posteriori} show that $E$ in fact is a direct sum of its Harder-Narasimhan factors.

On the other hand, the instability of $E$, which we assumed in \pref{th:int1}, follows almost automatically from the bigness of $- K _{ X }$.

\begin{theorem}[$=$\pref{cr:sss_implies_big}]\label{th:int2}
Let $C$ be as in \pref{th:int1}, and $E$ be a vector bundle of rank $r$
on $C$ such that
$
 - K _{ \bP _{ C } \lb E \rb }
$
is big. Then $E$ is not strongly semi-stable; namely, the pull-back of $E$
by an iteration of the Frobenius morphism of $C$ is unstable.
In particular,
$E$ itself is unstable if
$
 \chara ( \bfk ) = 0
$
or $ g = 1 $.
\end{theorem}
We remark that \pref{th:int1} and \pref{th:int2} follow also from Nakayama's work \cite[Chapter IV. 3.7. Lemma]{MR2104208}, which gives a criterion for the pseudo-effectiveness of line bundles on projective bundles based on a numerical criterion for semi-stability studied in \cite{MR946247} (see \pref{rm:nakayama's_results} for details.).
Compared to this, our arguments are more elementary.

An example of a curve of genus $g \ge 2$ and a rank $2$
semi-stable vector bundle on it which is not strongly semi-stable is constructed in \cite[Theorem 1]{MR0325616}.
The anti-canonical divisor of the projective bundle associated to the example, however, does not
seem to be big.
Hence the following subtle question remains open.

\begin{question}
Is $E$ itself unstable under the assumption of \pref{th:int2}?
\end{question}

Summing up the two results above, we obtain the following corollary.

\begin{corollary}\label{cr:main_cor}
\ 
\begin{enumerate}
\item
Suppose either
$
 \chara ( \bfk ) = 0
$
or $ g = 1 $. Then
$
 - K _{ X }
$
is big if and only if
$
 E \simeq L \oplus M
$
for some line bundles $L$ and $M$ on $C$ such that
$
 \deg L - \deg M > 2 g - 2
$.

\item
Suppose that
$
 \chara ( \bfk ) > 0
$
and
$
 g > 1.
$
If
$
 - K _{ X }
$
is big, then there exists
$
 e \ge 0
$
such that
$
 \lb F ^{ e } \rb ^{ * } E \simeq L ' \oplus M '
$
for some line bundles $L '$ and $M '$ on $C$ such that
$
 \deg L ' - \deg M ' > 2 g - 2
$,
where $F$ is the Frobenius map of $C$.
\end{enumerate}
\end{corollary}

The second aim of this paper is to show the following version of the base point free theorem as an application of \pref{cr:main_cor}.

\begin{theorem}[$=$\pref{cr:adj}]\label{th:adj}
Let $X$ be a normal projective surface over a field $\bfk$ which admits an $\bR$-divisor $\Delta$ such that
$\lfloor \Delta \rfloor = 0$
and
$
 - ( K_X + \Delta )
$
is a nef and big $\bR$-Cartier divisor.
Then $\Pic ( X )$ is a free abelian group of finite rank.
\end{theorem}

The base point free theorem for surfaces in positive characteristic is shown in \cite[Theorem 3.2]{MR3383599} for numerically \emph{non}-trivial nef divisors. \pref{th:adj} is nothing but the base point free theorem for numerically trivial divisors. This is partially shown in \cite[Corollary 3.6]{MR3383599} under the assumption that $X$ admits only rational singularities. In this case $X$ is necessarily $\bQ$-factorial, so one can use the minimal model program to obtain the conclusion. A typical example which is not covered there is the cone over a non-singular plane cubic curve.
By combining \pref{th:adj} and \cite[Theorem 3.2]{MR3383599}, we obtain the following complete form of the base point free theorem for normal projective surfaces.

\begin{theorem}[Base point free theorem]
Let $X$ be a normal projective surface over a field $\bfk$, $\Delta$ an $\bR$-divisor such that
$\lfloor \Delta \rfloor = 0$, and $D$ be a nef Cartier divisor
such that
$
a D- (K_X + \Delta )
$ 
is a nef and big $\bR$-Cartier divisor for some 
$
a \in \bZ_{>0}
$.
Then there exists
$
b_0 \in \bZ_{ > 0 }
$ 
such that $|bD|$ is base point free for any integer $b \ge b_0$.
\end{theorem}

In fact \pref{th:adj} for the case when $X$ admits a non-rational singularity is covered by \cite[Theorem 2.2 (iii)]{MR1821186} as well.
Hence we can also show \pref{th:adj} by combining it with \cite[Theorem 3.2]{MR3383599} (see \pref{sc:combination}).
Our method, however, allows us to investigate the following question concerning the boundary case.

\begin{question}
Does the conclusion of \pref{th:adj} hold if we allow some of the coefficients of 
$\Delta$
to be 1?
\end{question}
If we assume that
$
 - \lb K _{ X } + \Delta \rb
$
is ample, then our proof based on \pref{cr:main_cor} works equally well and we get the same conclusion.
However, the assertion of \pref{th:adj} does not necessarily hold when
$
 - \lb K _{ X } + \Delta \rb
$
is assumed to be only nef and big (see \pref{eg:counter}).
In \pref{sc:Delta=1} we investigate when the assertion of \pref{th:adj} fails,
by closely examining the proof of \pref{cr:adj}. For example, it is shown that $X$ needs to be birationally equivalent to the product of $\bP ^{ 1 }$ and a curve of genus one for the failure.

In the last section we construct some examples of non-minimal ruled surfaces with big anti-canonical divisor, by blowing up geometrically ruled surfaces along certain configurations of points.

\subsection*{Acknowledgements}

The authors are indebted to Hiromu Tanaka for the very useful discussion and
informing them of the current state of the art of the base point free theorem on surfaces. They also thank Kazuhiro Konno for informing them of the works by Noboru Nakayama and Yoichi Miyaoka.
S.~O.~ was partially supported by Grants-in-Aid for Scientific Research
(16H05994,
16K13746,
16H02141,
16K13743,
16K13755,
16H06337)
and the Inamori Foundation.

\subsection*{Notation and conventions}
The ground field, which will be assumed to be
algebraically closed of arbitrary characteristic unless otherwise stated,
will be denoted by $\bfk$.
A \emph{curve} (respectively, \emph{surface}) is a projective and geometrically integral scheme over $\bfk$ of dimension one (resp. two). For a vector bundle ($=$locally free sheaf) $E$ on a scheme $C$, the \emph{associated projective bundle} will be defined as
$
 \bP _{ C } \lb E \rb = \Proj _{ C } \lb \Sym E \rb
$.
An $\bR$-Cartier divisor $D$ on a projective scheme is said to be \emph{nef}
(respectively \emph{big}) if the intersection number with any integral curve is
non-negative (resp. if there is a positive constant $c \in \bR_{>0} $ such that
$
 \dim_ \bfk H^0(\lfloor k D \rfloor ) > c k ^2
$
holds for any sufficiently large integer $k$).
We will use the shorthand notation
$
 \hom =\dim \Hom, h ^{ i } = \dim H ^{ i },
$
etc.

\section{Stability of vector bundles on curves} 

We prove some facts about stability of bundles on curves.
Throughout this section $C$ is a smooth projective curve, $E$ is a vector bundle of rank $r$ on it, and $\pi \colon X \coloneqq \bP _{ C } (E) \to C$ is the projective bundle associated to $E$.
\begin{definition}\label{Def:stab}
The \emph{slope} of $E$ 
is defined by
\begin{align}
 \mu \lb E \rb \coloneqq \frac{\deg E}{ \rank E}.
\end{align}
$E$ is said to be \emph{semi-stable}
(respectively, \emph{stable}) if for any subsheaf
$
 0 \subsetneq V \subsetneq E
$
the inequality
\begin{align*}
 \mu(V) \le \mu(E) \quad ( \mbox{resp.} \  < )
\end{align*}
holds. $E$ is \emph{strongly semi-stable} if the vector bundle
\begin{align*}
E^{(r)} \coloneqq (F^r)^* E =
(\overbrace{F\circ \dots \circ F} ^{r \ \mathrm{times}})^*E 
\end{align*}
over $C$ is semi-stable for all $r\ge 0$, where 
$
F \colon C \to C
$ 
is the Frobenius morphism.
\end{definition}
Note that if $\chara(\bfk)=0$, there is no difference between semi-stability and strong semi-stability.

\begin{proposition}\label{pr:stable_property}
\ 
\begin{enumerate}
\item\label{it:symm_inherits_stability}
If $E$ is strongly semi-stable, 
the $n$-th symmetric power $ S ^{ n } ( E ) $ is also strongly semi-stable.

\item
Assume $g(C) \le 1$. Then $E$ is semi-stable if and only if it is strongly semi-stable.

%
\end{enumerate}
\end{proposition}

\begin{proof}
(1) is the consequence of \cite[Theorem $7.2$]{MR1488349}.
See \cite[THEOREM 2.9]{MR2483939} or \cite[Theorem 2.1]{MR714755} for (2).
\end{proof}

\begin{proposition}\label{pr:hom}
Let $E$ and $F$ be stable vector bundles of the same slope.
Then every non-zero map $f \colon E \to F$ is an isomorphism and
$ \hom (E,F)\le 1$.
\end{proposition}

\begin{proof}
This is an easy consequence of \cite[Proposition 5.3.3]{MR1428426} and 
\cite[Corollary 5.3.4]{MR1428426}. 
\end{proof}

\begin{proposition}\label{pr:stability_g=2}
Suppose
$
 g ( C ) \ge 2
$
and
$
 H ^{ 0 } ( X, \cO _{ X } \lb - m K _{ X } \rb ) \neq 0
$
for some positive integer $m$. Then $E$ is not strongly semi-stable.
\end{proposition}
 
\begin{proof} 
Recall that the canonical sheaf of $X$ admits the following isomorphism.
\begin{align}
 \cO _{ X } (K_X) \simeq
 \cO _{ X } ( - r ) \otimes \pi ^* ( \cO _{ C } ( K _{ C } ) \otimes \wedge ^{ r } E)
\end{align}
Hence a non-zero global section of
$
 \cO _{ X } \lb - m K _{ X } \rb
$
corresponds to a non-trivial morphism
\begin{align}
 \pi^*(\cO _{ C }(mK_C) \otimes (\wedge^r E)^{\otimes m}) \to \cO _{ X } (rm),
\end{align}
whose adjoint is a non-trivial morphism
\begin{align}
 \cO _{ C } ( m K _{ C } ) \otimes (\wedge^r E)^{\otimes m} \to \pi_*\cO _{ X } ( r m ) \simeq S ^{ r m } ( E ),
\end{align}
which necessarily is an injective morphism.
On the other hand, the slopes of these two bundles are
$
 m ( 2 g - 2 + \deg E )
$
and
$
 m ( \deg E )
$, respectively
since we can easily compute that
\begin{align*}
 \rank (S ^{ n } ( E ))= \binom{r-1+n}{n},\ 
 \operatorname{and} \ 
 \deg (S^n (E)) = \binom{r-1+n}{r} \deg E
\end{align*}
for any $ n\in \bZ_{\ge 0}$.
 Hence $S^{rm}(E)$ is unstable. By \pref{pr:stable_property} (\pref{it:symm_inherits_stability}), we see that $E$ is not strongly semi-stable.
\end{proof}

We next discuss the case when $g ( C ) = 1$.

\begin{proposition}\label{pr:stability_g=1}
Suppose
$
 g \lb C \rb = 1
$.
If $-K_X$ is big, then $E$ is not strongly semi-stable.
\end{proposition}

\begin{proof}
It is sufficient to show that $S^{rm}(E)$ is unstable for some $m$ by \pref{pr:stable_property}.
We will assume that $S^{rm}(E)$ is semi-stable for any $m$ and show
that $- K _{ X }$ is not big. 
Note that
$g(C) = 1 $ implies $\cO_C(K_C) \simeq \cO_C$,
and hence there exists the isomorphism
\[
 \pi _{ * } \cO _{ X } ( - m K_X )
 \simeq
 S^{rm}(E) \otimes L, 
\]
where $L = (\wedge ^r E)^{-m} $.
This implies that it is enough to show the convergence
\begin{align}
 \lim_{m \to \infty} \frac {h^0(C, S^{rm}(E) \otimes L) }{m^r} = 0
\end{align}
since $h^0(X, \cO(mK_X)) = h^0(C, \pi_* \cO(mK_X))$.

Let
\begin{align}
 0 \subset F_1 \subset  F_2 \subset \dots \subset F_{\ell} = S^{rm}(E) \otimes L
\end{align}
be a Jordan-H\"older filtration of $S^{rm}(E) \otimes L$,
so that for each $i$ there is an exact sequence 
\begin{align}\label{ex:T_i}
 0 \to F_{i-1} \to F_i \to  \gr _i \to 0 
 \end{align}
such that $ \gr _i$ is stable of slope $\mu( \gr _i) = \mu (S^{rm}(E) \otimes L)=0$. Note that
$
 \ell \le \rank (S^{rm}(E) \otimes L) = \binom{r-1+mr}{mr}
$.
Let $\xi \in \Ext^1( \gr _i,F_{i-1})$ be the extension class of \eqref{ex:T_i}, and consider the affine line
\begin{align}
 B \coloneqq \bfk \xi \subset \Ext ^{ 1 } \lb gr _{ i }, F _{ i - 1 } \rb.
\end{align}
Then we can construct the tautological vector bundle $\cE$ on $C \times B$ satisfying
\[
\cE |_{C \times \{t\} } \simeq
\begin{cases}
 F_i & t \neq 0 \in B\\
 F_{i-1} \oplus  \gr _i & t = 0 \in B
\end{cases}\]
(see \cite[Section (7.3)]{MR1428426}).
Let $p_2$ be the natural projection from $C \times B$ to B.
Then by taking a general $t \in B$, we obtain the following inequality
\[
 h^0(C, F_i)
 = h^0 (f^{-1}(t), \cE | _{ C \times \lc t \rc })
 \le h^0 ( f^{-1}(0), \cE | _{ C \times \lc 0 \rc } )
 = h^0 (C, F_{i-1}\oplus  \gr _i)
\]
by the upper semicontinuity theorem
(see \cite[Section 5 COROLLARY 1]{MR2514037}).
This inequality holds for all $i$, and we have
\begin{align}\label{eq:jj}
 h^0(C, S^{rm}(E) \otimes L) \le 
 h^0 \lb C, \bigoplus _{i=0} ^ \ell  \gr _i \rb
 = \sum _{i=0} ^ \ell h^0 (C,  \gr _i).
\end{align}
Since
$
 h^0 (C,  \gr _i) =\dim_\bfk \Hom (\cO _{ X },  \gr _i) \le 1
$ by \pref{pr:hom}, we obtain the inequality
\[
 \sum_{i=0} ^ \ell h^0 (C,  \gr _i) \le \binom{r-1+mr}{mr}.
\]
Thus we proved
$\displaystyle \lim_{m \to \infty} \frac {h^0(C, S^{rm}(E) \otimes L) }{m^r} = 0.$
\end{proof}

\begin{remark}
$E$ is in fact unstable because of \pref{pr:stable_property}.
\end{remark}

\begin{corollary}\label{cr:sss_implies_big}
$E$ is not strongly semi-stable if $-K_X$ is big.
\end{corollary}

\begin{proof}
This follows from \pref{pr:stability_g=2} and \pref{pr:stability_g=1}.
\end{proof}

\section{Geometrically ruled surface with big anti-canonical line bundle}

In the rest of this paper, let $C$ be a smooth projective curve of genus
$
 g \ge 1
$
and $E$ be a vector bundle on $C$ of rank $r$ unless otherwise stated. Let
$
 \pi \colon X = \bP _{ C } \lb E \rb \to C
$
be the projective bundle associated to $E$.
In this section we give a criterion for bigness of
$
 - K _{ X }
$
in the case where $E$ is a direct sum of line bundles.

First we consider vector bundles which are isomorphic to direct sums of line bundles.


\begin{proposition}\label{pr:main_theorem_split_case}
Consider the direct sum
$E = \bigoplus _{i=1} ^{r} L_i $
of line bundles $L_i$ on $C$ satisfying
$
 \deg L _{ r } \le
 \deg L _{ r - 1 } \le
 \cdots \le
 \deg L _{ 1 }.
 $
 Then $-K_X$ is big if and only if 
\begin{align}\label{eq:higher_rank_condition}
 (r-1) \deg L_1 - \sum _{i=2} ^{r} \deg L_i > 2g-2.
\end{align}
\end{proposition}

\begin{proof}
Note first that
\begin{align*}
\begin{split}
 h^0 (X, \cO_X(-mK_X)) = h^0 ( C, \pi_* \cO_X(-mK_X) ) \\
 = h^0 ( C , S^{rm} (E) \otimes (\wedge^r E \otimes \cO_C (K_C))^ {-m})
 = \sum _{k \in A_m}  h^0 (C, \cL_k ),
\end{split}
\end{align*}
where
\begin{align}
\begin{split}
 A_m \coloneqq \lc k = (k_1, \dots ,k_r) \in \bZ ^{r} _ {\ge 0} \mid \sum  k_i = rm \rc,\\
 \cL_k \coloneqq L_1 ^{(k_1-m)} \otimes \dots \otimes L_r ^{(k_r - m)} \otimes \cO_C (-mK_C).
\end{split}
\end{align}

Assume \pref{eq:higher_rank_condition}.
We may assume $\deg L_i \ge 0 $ without loss generality by replacing $L_i$ with $L_i \otimes L_r$.
Take a sufficiently small rational number
$ 
\delta >0 
$
such that
\[
a \coloneqq  (1-\delta) (r-1) \deg L_1 - \sum _{i=2} ^{r} \deg L_i - (2g-2) > 0.
\]
If 
$
0 \le k_i \le \delta m
$
for all
$
i \ge 2
$,
then
$
 (r - (r-1) \delta) m \le k_1 \le rm
$
and hence
\begin{align*}
h^0 (\cL_k) = \sum_{i=1} ^{r} (k_i - m) \deg L_i -m (2g-2)+ 1 - g + h^1 (\cL_k) \\
\ge ( (r - (r-1) \delta) m - m ) \deg L_1  -m \sum_{i \ge 2} \deg L_i -m (2g-2) -g \\
\ge m \lb (1-\delta) (r-1) \deg L_1 - \sum _{i=2} ^{r} \deg L_i - (2g-2) \rb - g \\
\ge am -g.
\end{align*}
Since
$
\# \lc k \in A_m \mid 0 \le k_i \le \delta m \quad \forall i \ge 2 \rc \ge (\delta m - 1) ^ {r-1}
$,
we deduce 
\[
h^0 (X, \cO_X (-mK_X)) \ge (am-g) (\delta m-1) ^{r-1} = a s^{r-1} m^r + O (m^{r-1}).
\]
Hence $-K_X$ is big.

We next prove the converse.
%
Suppose that \pref{eq:higher_rank_condition} does not hold.
We may assume that
$
0 \le \deg L_1 
$
and
$
\deg L_i \le 0
$
for all
$
i \ge 2
$
without loss of generality.
Note that if $h^0(\cL_k)>0$ then
\begin{align}\label{eq:dd}
 h^1(\cL_k) = h^0(\cO_C(K_C) \otimes \cL_k ^{ - 1 }) \le h^0(\cO_C(K_C)) = g.
\end{align}
Hence for all $k \in A_m$, either
$
 h ^{ 0 } \lb D _{ i } \rb = 0
$
or

\begin{align*}
h^0 (\cL_k) = \sum_{i=1} ^{r} (k_i - m) \deg L_i -m (2g-2)+ 1 - g + h^1 (\cL_k) \\
\le (k_1 -m) \deg L_1 + \sum_{i=2} ^{r} (k_i - m) \deg L_i -m (2g-2) + 1 + h^1 (\cL_k) \\
\le (rm -m) \deg L_1 - m \sum_{i=2} ^{r} \deg L_i -m (2g-2) + 1 + h^0 (K_X) \\
\le m \lb (r-1) \deg L_1 - \sum _{i=2} ^{r} \deg L_i - (2g-2) \rb + 1 + h^0 (K_X) \\
\le 1 + h^0 (K_X).
\end{align*}

Then $-K_X$ is not big because
$
 \# A_m \le (rm)^{r-1}
$.
\end{proof}

Let $r =2$ and
consider an exact sequence of vector bundles on $C$ as follows,
where $L$ and $M$ are both line bundles.
\begin{align}\label{eq:extension}
0 \to L \to E \to M \to 0.
\end{align}
Below is the main theorem of this section.
\begin{theorem}\label{th:main_theorem}
Let
$
 E
$
be an unstable vector bundle of rank $2$ on $C$. 
Then $-K_X$ is big if and only if there are line bundles
$L$ and $M$ on $C$ such that 
$
 E \simeq L \oplus M
$
and
$
 \deg L- \deg M > 2 g - 2
$.
\end{theorem}

\begin{proof}
The 'if' direction is already shown in \pref{pr:main_theorem_split_case}.
To show the 'only if' direction, assume that $-K_X$ is big. We then obtain an exact sequence \eqref{eq:extension} 
such that $\deg L > \deg M$ as the Harder-Narasimhan filtration of $E$.
Let $\xi \in \Ext^1(M,L)$ be the corresponding extension class and set
$
 B \coloneqq \bfk \xi \simeq \bA ^1
$.
Then we have a rank $2$ vector bundle $\cE$ on $C \times B$ such that
\begin{align}
\cE |_{C \times \{t\} } \simeq
\begin{cases}
 E & ( t \neq 0 )\\
 L \oplus M & ( t = 0 ).
\end{cases}
\end{align}
Consider the following diagram.
\begin{align}
 \xymatrix{
    \cX:=\bP_{C \times B}(\cE) \ar[r]^-{\pi} \ar[dr]_{f} & C \times B \ar[d]^{p_2} \\
    &  B 
 }
\end{align}
Note that
$
 \cX _{ 0 } \simeq \bP_C( L \oplus M)
$
and
$
 \cX _{ t } \simeq \bP_C(E)\ (t \neq 0)
$.
Furthermore we have
\[
 \cO_\cX(K_{\cX})|_{ \cX _{ t } }=\cO_\cX(K_{ \cX _{ t } })\ \ (t \in B),
\]
where $\cO_\cX (K_{\cX})$ is the canonical sheaf of $\cX$.
By using the upper semicontinuity theorem \cite[2.5 COROLLARY 1]{MR2514037} we obtain the inequality
\begin{align*}
 h^0( \cX _{ t }, \cO_\cX( m K_{\cX})| _{ \cX _{ t } } )
 \le
 h^0 ( \cX _{ 0 }, \cO_\cX( m K_{\cX}) | _{ \cX _{ 0 } }) 
\end{align*}
for any $ m > 0 $ and general $t \in B$. Hence we obtain the inequality
\begin{align}
 h ^{ 0 } \lb X, \cO _{ X }(- m K _{ X }) \rb
 \le
 h ^{ 0 } \lb \bP _{ C } \lb L \oplus M \rb, \cO_{ \bP _{ C } \lb L \oplus M \rb }( - m K _{ \bP _{ C } \lb L \oplus M \rb }) \rb
\end{align}
for any
$
 m > 0
$, so that the bigness of $-K_{X}$ implies that of
$
 - K _{\bP_C (L \oplus M )}
$.
Thus we obtain the inequality
$
 \deg L - \deg M > 2 g - 2
$
from \pref{pr:main_theorem_split_case}.
This in fact implies that \eqref{eq:extension} is a trivial extension, since
\begin{align}
 \ext _{ C } ^{ 1 } \lb M, L \rb
 = h^1 \lb C, L \otimes M ^{ - 1 } \rb
 = h^0 \lb C, \cO _{ C } ( K _C ) \otimes L^{-1} \otimes M \rb = 0.
\end{align}
For the last equality, use
$
 \deg \lb \cO _{ C } ( K _C ) \otimes L^{-1} \otimes M \rb < 0
$.
Hence
$
 E \simeq L \oplus M
$,
concluding the proof.
\end{proof}

\begin{remark}\label{rm:nakayama's_results}
In fact \pref{th:main_theorem} and \pref{cr:sss_implies_big} are special cases of \cite[Chapter IV. 3.7. Lemma]{MR2104208}.
Let
$
\pi \colon X \coloneqq \bP _{ C } (E) \to C
$ 
be the projective bundle. Consider the Harder-Narasimhan filtration of $E$
\begin{align}
0 \subset E_1 \subset E_2 \subset \dots \subset E_k = E,
\end{align}
and assume that the successive quotients
$
 \gr _i \coloneqq E_i / E_{i-1}
$
are strongly semi-stable.
By combining \cite[Chapter IV. 3.7. Lemma]{MR2104208} with the well-known fact that the big cone is the interior of the pseudo-effective cone, we see that
the line bundle $\cO_{X} (m) \otimes \pi^* D$ on $X$,
where
$
m \in \bZ
$
and
$D$ is a line bundle on $C$, is big if and only if
\begin{align}\label{eq:big_slope}
 \deg D < m \mu ( \gr _1).
\end{align}

This implies that if $E$ is strongly semi-stable, $-K_X$ cannot be big as long as $g(C) \ge 1$.
Under the assumption of \pref{th:main_theorem}, it follows that
$-K_X$ is big
if and only if
$
 \deg L - \deg M > 2 g - 2
$,
since $L =  \gr _1$ and $M =  \gr _2$.

The proof of \cite[Chapter IV. 3.7. Lemma]{MR2104208}, in turn, is based on a numerical criterion for semi-stability \cite[Theorem 3.1]{MR946247}. Compared to this, our arguments above are more elementary.
\end{remark}

\begin{remark}
If the rank of an unstable vector bundle $E$ is at least three,
the bigness of $-K_X$ does not necessarily imply that
$E$ is isomorphic to a direct sum of semi-stable vector bundles.
To give such an example, let $g \ge 2$,
$L$ be a line bundle of degree $g$ on $C$ such that $h^0 (C, \omega_C \otimes L^{-1}) >0 $,
and
$E' \coloneqq \cO_C \oplus \cO_C $.
Note that $E'$ is semi-stable.
Let $E$ be the extension of
$E'$ by $L$ given by a non-trivial element of
\begin{align*}
\Ext ^1 (E', L) 
\simeq  H^1 (C, ( E ' ) ^{ - 1 } \otimes L) \\
\simeq H^0 (C, \omega_C \otimes ( ( E ' ) ^{ - 1 } \otimes L)^{-1}) ^{ \vee }\\
\simeq H^0 (C, \omega_C \otimes ((\cO_C \oplus \cO_C) \otimes L)^{-1} ) ^{ \vee } \\
\simeq H^0 (C, \omega_C \otimes L^{-1}) ^{ \vee } \oplus H^0 (C, \omega_C \otimes L^{-1}) ^{ \vee } \neq 0.
\end{align*}
By the same arguments in \pref{rm:nakayama's_results}, we see that $ -K_{\bP_C(E)} $ is big. To see that $E$ does not admit a decomposition into semi-stable vector bundles, one can use the uniqueness of the Harder-Narasimhan filtration.
Note that if $g = 1$, there are no such examples (see \cite[Theorem 10]{MR2264108}).
\end{remark}

\begin{corollary}\label{cr:structure_of_Frobenius_pullback}
Let
$
 E
$
be a vector bundle of rank $2$ on $C$.
If
$
 - K _{ X }
$
is big, there exists an integer
$
 e \ge 0
$
such that the Frobenius pull-back
$
 \lb F ^{ e } \rb ^{ * } E
$
is isomorphic to a direct sum of line bundles
$
 L ' \oplus M '
$
such that
$
 \deg L ' - \deg M ' > 2 g - 2
$.
\end{corollary}

\begin{proof}
By \pref{cr:sss_implies_big}, we see that $E$ is not strongly
semi-stable. Hence there exists a positive integer
$
 e \ge 1
$
such that the Frobenius pull-back
$
 E ' \coloneqq \lb F ^{ e } \rb ^{ * } E
$
is unstable. Set
$
 f \coloneqq F ^{ e }
$,
$
 Y ' \coloneqq \bP_C ( E ' )
$, and name the morphisms as in \pref{fg:first_diagram} below.

By applying the exact functor
$- \otimes _{ C } \pi_* \cO_Y (-m K_Y)$
to the injective homomorphism
$
 \cO_C \hookrightarrow f_* \cO_C
$
and then taking $H ^{ 0 }$, we obtain the inequality
\begin{align}
 h^0(\pi_* \cO_Y (-m K_Y)) \le h^0(f^*( \pi_* \cO_Y (-m K_Y))).
\end{align}
On the other hand, for each integer $m > 0$ we have 
\begin{align}
 \pi'_* ( \cO_{Y'}(- m K _{ Y ' }) ) =  f^*( \pi_*\cO_{Y'} (-m K_Y)) \otimes
 (f^* \omega_C \otimes \omega_C^{-1}) ^{ \otimes m }.
\end{align}
Since
$
 \deg (f^* \omega_C \otimes \omega_C^{-1}) > 0
$,
we have
$
 H ^{ 0 } \lb \lb f^* \omega_C \otimes \omega_C^{-1} \rb ^{ \otimes m }\rb
 \neq 0
$
for any sufficiently large $m$. By the similar arguments as above,
we obtain the following inequality.
\begin{align}
 h^0(f^*( \pi_* \cO_Y (-m K_Y))) \le h^0( \cO_{Y'}(- m K _{ Y ' } ))
\end{align}
Thus we see that $- K _{ Y ' }$ is big, and the claim now follows from \pref{th:main_theorem}.
\end{proof}


%


\section{Numerically trivial line bundles on normal surface with nef and big anti-canonical bundle}
\label{sc:Numerically trivial line bundles on normal surface with nef and big anti-canonical bundle}

In this section we prove that a numerically trivial line bundle on a
normal projective surface with nef and big anti-canonical divisor
is trivial. This is a result which is proved for some partial cases in
the preceding works \cite{MR3319838, MR3383599}.
In fact we give two different proofs.
In the first subsection, we give a proof as an application of our results
obtained so far. The point is that the minimal model
(in the classical sense) of the minimal resolution has big anti-canonical
line bundle, so that we can apply our results to it.
In the second subsection, we give a proof by combining known results.

%
%
%
\subsection{Proof via \pref{cr:structure_of_Frobenius_pullback}}

\begin{theorem}\label{th:adjunction}
Let $X$ be a normal projective surface 
such that $-(K_X+\Delta)$ is a nef and big \bR-Cartier divisor,
where $\Delta$ is an \bR-divisor satisfying $\lfloor \Delta \rfloor = 0$.
Then there exists a positive integer $m > 0$ which depends only on $X$ such that
any numerically trivial line bundle $\cL$ on $X$ satisfies
$
 \cL ^{ \otimes m } \simeq \cO _{ X }
$.
\end{theorem}  

\begin{proof}
If
$
 X
$
is rational, the assertion is obvious. Suppose otherwise.
Let $\varphi \colon {\Xtilde} \to X$ be the minimal resolution of $X$ and set
$\Delta' \coloneqq \varphi^{-1}_* \Delta$.
Then we have the canonical bundle formula as follows.
\begin{align}\label{eq:22}
 K_{\Xtilde}+\Delta' \equiv \varphi^* (K_X + \Delta) + \sum a_i E_i
\end{align}
Since $ a_i \le 0$ (see \cite[Corollary $4.3$]{MR1658959}) and
$- ( K_X + \Delta )$ is big, so is $-K_{\Xtilde}$.

\begin{figure}
\begin{align}
 \xymatrix{
   & {\Xtilde} \ar[r]^{\varphi} \ar[d]^{\varepsilon} &  X\\
 Y' \ar[r]^{f'}\ar[d]^{\pi'}  &  Y\ar[d]^{\pi} \\
 C\ar[r]^{f}  &  C
  }
\end{align}
\caption{\label{fg:first_diagram}}
\end{figure}

\begin{figure}
\begin{align}
 \xymatrix{
 & {\Xtilde}\ar[r]^{\varphi} & X\\
(\Ctilde)_{\norm}\ar[r]^{\psi} 
& \Ctilde\ar[r]^{\varphi} \ar@{^{(}->}[u]\ar[d]^{\varepsilon}  &\{pt\} \ar@{^{(}->}[u]  \\
s(C)\ar[u]^{\alpha} \ar@<0.5ex>[d]^-{\pi'} \ar[r]^{f'} & 
\Cbar \ar[d]^{\pi}\\
C \ar@<0.5ex>[u]^-s  \ar[r]^{f} & C
}
\end{align}
\caption{\label{fg:main_diagram}}
\end{figure}

Let
$
 \varepsilon \colon \Xtilde \to Y
$
be a composition of contractions of $(- 1)$-curves such that
$Y$ is a geometrically ruled surface $\bP_C(E)$ with
$
 g ( C ) \ge 1
$
(recall that $X$ is assumed to be non-rational).
Consider the canonical bundle formula
\begin{align}\label{eq:11}
K _{ \Xtilde } = \varepsilon^*  K_Y + F
\end{align}
with respect to $\varepsilon$, where $F$ is an effective
$\varepsilon$-exceptional divisor on $\Xtilde$.
Since $-K_{\Xtilde}$ is big, so is $-K_Y$.
By \pref{pr:main_theorem_split_case}, the pull-back
$
 E ' \coloneqq f ^{ * } E
$,
where
$
 f = F ^{ e }
$
is some iteration of the Frobenius morphism, is isomorphic to
a direct sum of line bundles
$
 L ' \oplus M '
$
such that
$
 \deg L ' - \deg M ' > 2 g - 2
$.

Let $s$ be the section of $\pi'$ corresponding to the natural projection $E' \to M'$.
Let
$
 \Cbar \subset Y
$
be the closed subscheme
$
 f ( s ( C ) ) \subset Y
$
equipped with the reduced structure, and
$\Ctilde$ be the strict transform of $\Cbar$ in ${\Xtilde}$.
These varieties and morphisms form the diagram
\pref{fg:main_diagram}.

\begin{claim}\label{cl:positivity}
$(\Ctilde. K_{\Xtilde}+\Delta ') > 0$.
\end{claim}
\begin{proof}
It is easy to see that
$
 f' _*[s(C)] = \lb \deg f'| _s(C) \rb [\Cbar]
$
and
$(\cO _{ Y ' }(1).s(C)) = \deg M'$. Hence
\begin{align}
\begin{aligned}
 \lb \deg f' | _C \rb (\Cbar. K_Y) = s(C).f'^*K_{Y} \\
= s(C).f'^* \cO _{ Y } (-2) + s(C). f'^* \pi^* (\wedge ^2 E \otimes\ \omega_{C}) \\
= s(C).\cO _{ Y ' }(-2) + s(C).\pi'^* (\wedge ^2 E' \otimes f^* \omega_{C})\\
= \deg L' - \deg M' + \lb \deg f \rb (2g-2) > 0,
\end{aligned}
\end{align}
so that
$
 (\Cbar. K_Y) > 0
$.
If $\Ctilde ^2\ge 0$, it follows from this inequality and \pref{eq:11}
that
$(\Ctilde. K_{\Xtilde}+\Delta') > 0$.
Now suppose $\Ctilde ^2<0$. Write
\[
 \Delta' = \alpha \Ctilde + R,
\]
where $0 \le \alpha <1$.
Then since $g(\Ctilde) \ge 1 $ and $\Ctilde ^2 < 0$, we have 
\begin{align}
\begin{aligned}\label{neq:intersection}
(\Ctilde. K_{\Xtilde}+\Delta) = (\Ctilde^2) + (\Ctilde.K_{\Xtilde}) + (\alpha - 1)(\Ctilde ^2) + (\Ctilde.R) \\
=2g(\Ctilde)-2 + (\alpha-1)(\Ctilde^2) + (\Ctilde.R) >0.
\end{aligned}
\end{align}
\end{proof}

Consider the following equality obtained from \pref{eq:22}.
\[ 
(\Ctilde.K_{\Xtilde}+\Delta') = (\varphi_*(\Ctilde).K_X+\Delta) + \sum a_i(E_i. \Ctilde )
\]
Since $- (K _{ X }+\Delta)$ is nef, $(\varphi_*(\Ctilde).K_X+\Delta) \le 0$.
If $\Ctilde$ is not contracted by $\varphi$, then the second term of the right hand side is $\le 0$ since $a_i \le 0$ and $\Ctilde \neq E _{ i }$ for all $i$.
This contradicts \pref{cl:positivity}.
Hence $\Ctilde$ must be contracted to a point by $\varphi$.

Now we prove the theorem. Let $\cL$ be a numerically trivial line bundle on $X$.
There exists a numerically trivial line bundle $\cL_Y$ on $Y$ such that
$
 f ^* \cL = \varepsilon ^* \cL_Y
$.
Since
$
 \Pic  \bP _{ C } (E)  = \pi ^{ * } \Pic C \perp \bZ \cO ( 1 )
$,
there exists a numerically trivial line bundle
$\cL_C$ on $C$ such that $\cL_Y = \pi ^* \cL_C.$
It follows from \pref{fg:main_diagram} that
$
 (\varepsilon ^* \pi^* \cL_C) |_{\Ctilde} = \varphi^* \cL|_{\Ctilde} = 0
$.
Now let \[\psi: (\Ctilde)_{\norm} \to \Ctilde \] 
be the normalization of $\Ctilde$. 
Note that \[\varepsilon \circ \psi :(\Ctilde)_{\norm} \to \Ctilde \to \Cbar \]
is a normalization of $\Cbar$ as well.
These morphisms form the diagram \pref{fg:main_diagram}, where
$\alpha$ is the natural morphism induced by the universal property of the normalization. Then we have 
\begin{align}
 \cO _{ C } \simeq s^* \alpha^* \psi^* ((\varepsilon ^* \pi^*) \cL_C|_{\Ctilde})
 \simeq s^* f'^* \pi^* \cL_C
 \simeq f^* \cL_C \simeq \cL_C ^{ \otimes \deg f }.
\end{align}
Therefore we can take
$
 m \coloneqq \deg f
$
when
$
 \chara \lb \bfk \rb > 0
$,
and
$
 m \coloneqq 1
$
otherwise.
\end{proof}

We will use \cite[Lemma $2.1$]{MR1821186} in the proof of \pref{th:adj} below.

\begin{lemma}($=$\cite[Lemma 2.1]{MR1821186})\label{lm:freeness}
Let $X$ be a normal surface with irregularity
$
q(X) \coloneqq \dim \Pic^0 (X) = 0
$.
If no multiple $nK_Z$ with $n>0$ is effective, then $\Pic(Z)$ is a free abelian group of finite rank.
\end{lemma}

\begin{proof}
See \cite[Lemma 2.1]{MR1821186}.
\end{proof}

\begin{corollary}\label{cr:adj_closed}
Under the same assumption of \pref{th:adjunction}, $\Pic(X)$ is a free abelian group of finite rank.
\end{corollary}

\begin{proof}
If $X$ is rational, the assertion is obvious. Suppose otherwise.
It is enough to show that $q(X)=0$ by \pref{lm:freeness} and \pref{lm:non_effective} below.
Note that in this case 
$\Pic^0_X $ is smooth
since
$
 h ^2 ( \cO _{ X } )= 0
$
by \cite[PROPOSITION 9.5.19]{MR2222646}.

Moreover $\Pic^0_X$ is an abelian variety, 
because $\Pic^0_X$ is projective by \cite[THEOREM 9.5.4]{MR2222646}.
We can now apply \pref{th:adjunction} to $(X, \Delta)$ 
to obtain an integer $m>0$ such that 
$
L^{\otimes m} \simeq \cO_{X}
$ 
for all $L \in \Pic^0_X$.
Hence the $m$-th power map
\[ 
 \m \colon \Pic^0_X \to \Pic^0_X
\]
of the abelian variety $\Pic^0_X$ factors through the constant map to the identity.
On the other hand $\m$ is surjective and finite by generalities of abelian varieties; for example, it follows from \cite[p.59, Corollary 3]{MR2514037}
that the pull-back of an ample line bundle by $\m$ is ample.
Hence we see $q(X)=0$, concluding the proof.
\end{proof}

\begin{lemma}\label{lm:non_effective}
Suppose $-(K_X + \Delta)$ is big for some effective $\bR$-divisor $\Delta$.
Then $nK_X$ is not effective for all $n>0$.
\end{lemma}

\begin{proof}
Suppose
$
 n K _{ X } \sim D \ge 0
$
for a contradiction. 
Since
$
 - \lb K _{ X } + \Delta \rb
$
is big,
$
 - n (K_X + \Delta) \sim _{ \bR } D ' \ge 0
$
for some integer $n > 0$ and an effective $\bR$-divisor
$
 D ' \neq 0
$.
Summing up, we see that the non-zero effective $\bR$-divisor
$
 D + D ' + n \Delta
$
is a principal $\bR$-divisor.
This contradicts the fact that an effective principal $\bR$-divisor
on a $\bfk$-scheme $X$ such that
$
 H ^{ 0 } \lb X, \cO _{ X } \rb = \bfk
$
is never effective unless it is trivial.
\end{proof}

We now generalize \pref{cr:adj_closed} to arbitrary fields.
In the rest of this section, let
$
\bfk
$
be not necessarily algebraically closed and
$
 X
$
be a projective, geometrically integral, and geometrically normal surface
over $\bfk$.
Let $\overline{\bfk}$ be an algebraic closure of $\bfk$, and
\begin{align}
 \pi \colon \Xbar = X \otimes _\bfk \overline{\bfk} \to X
\end{align}
be the natural projection. For a coherent sheaf
$
 \cF \in \coh X
$, we use the shorthand notation
$
 \overline{\cF} \coloneqq \pi ^{ * } \cF
$. Then by the standard descent theory we have a canonical isomorphism
\begin{align}\label{eq:ext}
 H^i(X, \cF)\otimes_\bfk \overline{\bfk}
 \simeq 
 H^i(\Xbar, \overline{\cF} ).
\end{align}
In particular, $X$ has at worst rational singularities if and only if
$\Xbar$ does. Furthermore, note that if $L$ is a nef and big line bundle, 
then so is $\pi^*L$.

\begin{lemma}\label{lm:field_ext}
Let
$\bfk$ and $X$ be as above.
Let
$
\Delta
$
be as in \pref{th:adjunction} and assume that
any irreducible component of $\Delta$ is geometrically reduced.
Then the \bR-divisor $\Deltabar = \pi ^{ * } \Delta$ on $\Xbar$
also satisfies
$\lfloor \Deltabar \rfloor = 0$
and
$-(K_\Xbar + \Deltabar)$ is nef and big.
\end{lemma}

\begin{proof}
Assume that $\Delta = \sum a_i D_i$, where $D_i$ is a prime divisor on $X$.
Let $\Deltabar$ = $\sum a_i \Dbar _{ i }$ where 
$\Dbar _{ i } = D_i \otimes _\bfk \overline{\bfk} \subset \Xbar$. Since each $D_i$ is geometrically reduced and $\lfloor \Delta \rfloor = 0$, we have $\lfloor \Deltabar \rfloor = 0$. Then we only have to show that $-\pi^*(K_X + \Delta) = -(K_\Xbar + \Deltabar)$, but we can check this by restricting everything over the smooth locus of $X$.
\end{proof}



%

Now we are ready to prove \pref{th:adj}.

\begin{corollary}\label{cr:adj}
Under the same assumption of \pref{lm:field_ext}, $\Pic(X)$ is a free abelian group of finite rank.
\end{corollary}

\begin{proof}
Note that the natural homomorphism $\pi^* \colon \Pic (X) \to \Pic (\Xbar)$ is injective.
In fact, if $\pi ^*L = \cO_\Zbar$,
then $h^0(L) = h^0(L^{-1}) = 1$ and this implies that $L = \cO_Z$ since $H^0(Z, \cO_Z) = \bfk$. 
Hence we may assume that
$
 \bfk = \bar{\bfk}
$ by \pref{lm:field_ext}, and this is done in \pref{cr:adj_closed}
\end{proof}


%
%
\subsection{Proof via combination of known results}
\label{sc:combination}

We show the \pref{cr:adj} by combining the following
two known results.

\begin{proposition}\label{pr:rational}
Under the same assumption of \pref{lm:field_ext},
we also assume that $X$ has at worst rational singularities.
Then $\Xbar$ is a rational surface.
In particular, $\Pic(X)$ is a free abelian group of finite rank.
\end{proposition}

\begin{proof}
Apply 
\cite[Theorem $3.5$]{MR3383599} and \cite[Corollary 3.6]{MR3383599} to 
$(\Xbar, \Deltabar)$.
Then the injectivity of $\pi^*$ in the proof of \pref{cr:adj} implies the second part.
\end{proof}

\begin{proposition}\label{pr:non_rational}
Let
$\bfk$ and $X$ be as \pref{lm:field_ext}.
Assume $X$ has non-rational singularities and 
$nK_X$ is not effective for all $n>0$.
Then $\Pic(X)$ is a free abelian group of finite rank.
In particular if $-(K_X + \Delta)$ is big for some effective $\bR$-divisor $\Delta$
and $X$ has non-rational singularities,
$\Pic(X)$ is a free abelian group of finite rank.
\end{proposition}

\begin{proof}
The first part follows from \cite[Theorem 2.2]{MR1821186} and the injectivity of $\pi^*$
since
$
K_\Xbar \simeq \pi^ *K_X
$.
Then the second part follows from \pref{lm:non_effective}.
\end{proof}


\subsection{What if $\lfloor \Delta \rfloor \in \lc 0, 1 \rc$?}
\label{sc:Delta=1}

Consider the case when some of the coefficients of $\Delta$ is $1$; namely,
replace the assumption
$
 \lfloor \Delta \rfloor = 0
$
of \pref{cr:adj}
with the weaker condition
$
 \lfloor \Delta \rfloor \subset \lc 0, 1 \rc
$.
Even in this generality, if $-(K_X + \Delta)$ is ample, then \pref{th:adjunction} and \pref{cr:adj} are still true. The proof of \pref{th:adjunction} given above works in this generality without change.
On the other hand, they do not hold true if $-(K_X + \Delta)$ is only nef and big. Such an example is given in \pref{eg:counter} below.

In fact \pref{eg:counter} is typical in the following sense. Assume that $\Pic (X)$ is not a free abelian group of finite rank. Then by closely examining the proof of \pref{th:adjunction},
especially the computation \eqref{neq:intersection},
we can check that the following properties have to be satisfied.
\begin{itemize}
\item
$g(C)=1$.

\item
$\alpha=1$.

\item
$
 \Ctilde \cap \Supp R = \emptyset
$.
\end{itemize}
Furthermore it follows from \pref{pr:non_rational} that $X$ has at worst rational singularities. It would be interesting to classify these exceptional cases.

\begin{example}\label{eg:counter}
Let $X$ be the projective cone over a smooth plane cubic curve $C$ defined by an equation
$
 F \lb x, y, z \rb \in \bfk [ x, y, z ]
$.
Since $X$ is the singular hypersurface of
$
 \bP ^{ 3 } _{ x, y, z, w }
$
defined by $F$, we obtain
$
 -K_X = \cO_X( 1 )
$
by the adjunction formula \cite[Chapter II, Theorem 7.11]{MR0463157}.
As is well known, there is a birational contraction
\begin{align}
 \varepsilon \colon \Xtilde \coloneqq \bP _{ C } \lb \cO _{ C } \oplus \cO _{ C } ( 1 ) \rb\to X
\end{align}
which contracts the section
$
 E \subset \Xtilde
$
of $\pi$ corresponding to the quotient map
$
 \cO _{ C } \oplus \cO _{ C } ( 1 ) \to  \cO _{ C }
$
to the vertex of $X$.

By using the structure morphism
$
 \pi \colon \Xtilde \to C
$, we obtain
\begin{align*}
 \bR \Gamma \lb \Xtilde, \cO _{ \Xtilde } \rb \simeq \bR \Gamma \lb C, \bR \pi _{ * } \cO _{ \Xtilde } \rb
 \simeq
 \bR \Gamma \lb C, \cO _{ C } \rb \simeq \bfk \oplus \bfk [ - 1 ].
\end{align*}
Hence the Picard scheme of $\Xtilde$ is smooth of dimension one, so that $X$ admits a numerically trivial but non-trivial line bundle.

On the other hand, by applying the adjunction formula to the embedding
$
 E \hookrightarrow \Xtilde
$,
one can easily verify the equality
$
 - ( K _{ \Xtilde } +E ) = - \varepsilon ^{ * } K _X
$.
Note that the right hand side is nef and big, since
$
 - K _{ X } = \cO _{ X }( 1 )
$
is (very) ample. Hence the pair $(\Xtilde, E)$ is an example which does not satisfy the conclusion of \pref{th:adjunction}.
\end{example}

\section{Some examples of ruled surfaces whose anti-canonical sheaf is big}
In this section we construct some examples of (not necessarily geometrically) ruled surfaces whose anti-canonical sheaf is big. 
Note that such a surface is obtained from a geometrically ruled surface
(i.e. a minimal model in the classical sense) with big anti-canonical bundle by repeatedly blowing up smooth points. The examples below indicate that there are quite a few examples of such surfaces. This is comparable to the case of big rational surfaces \cite{MR2824848}. 

We first give a bigness criterion for line bundles on geometrically ruled surfaces which is associated to decomposable rank two vector bundles.

\begin{proposition}\label{pr:crit_big}
Let
$
 L
$
be a line bundle on
$
 C
$
with
$
\deg L \ge 0
$. 
Consider
$
 E \coloneqq L \oplus \cO _{ C }
$
and let
$
 \pi \colon X \coloneqq \bP_C(E) \to C
$
be the canonical projection. For a line bundle
$
 L _{ C }
$
on
$
 C
$,
the line bundle
$
 \cO _{ X } (n) \otimes  \pi^* L _{ C }
$
on
$
 X
$
is big if and only if $ n > 0 $ and
$
 n \deg L + \deg L _{ C } > 0
$.
\end{proposition}

\begin{proof}
This is proved by almost the same argument in the proof of \pref{pr:main_theorem_split_case} or
\pref{rm:nakayama's_results}. As we discussed in \pref{rm:nakayama's_results}, this is a special case of \cite[Chapter IV. 3.7. Lemma]{MR2104208}.
\end{proof}

Below is an immediate corollary of \pref{pr:crit_big}.

\begin{corollary}\label{cr:big}
Let
$
L, X
$
be as in \pref{pr:crit_big}. For a line bundle
$
 M
$
on
$
 X
$,
the line bundle
$
 \cO _{ X } (K_X) \otimes M^{-1}
$ 
is big if and only if either
$
 M \simeq \pi ^{ * } L _{ C }
$
for some line bundle
$
 L _{ C }
$
on
$
 C
$
such that
\begin{align}
 \deg L _{ C } + \deg L > 2 g - 2,
\end{align}
or
$
 M \simeq \cO _{ X } (1) \otimes \pi ^* L _{ C }
$
for some line bundle
$
 L _{ C }
$
on
$
 C
$
such that
\begin{align}
 \deg L _{ C } < - ( 2 g - 2 ).
\end{align}
\end{corollary}

Let $f \colon \Xtilde \to X $ be the blow-up of X at a finite set of points which is contained in the support of an effective divisor $D$.
If $-K_X - D$ is big, $-K_{\Xtilde}$ is also big.
In fact, the second term of the right hand side of the following equation is effective, and the first term is big.
\begin{align*} 
-K_{\Xtilde} = - f^* K_X - E  
= - f^* \lb K_X + D \rb + \lb  f^*D -E \rb.
\end {align*}
Furthermore, by the same argument, we can see that the bigness of the anti-canonical divisor is preserved under the successive blow-ups centered at points in the strict transforms of $D$.
Thanks to this observation, we obtain the following examples.

\begin{example}
Let
$
L, X
$
be as in \pref{pr:crit_big}, and assume moreover
$
\deg L > 2 g- 1
$.
Set
$
 k \coloneqq \deg L - (2g-1)
$
and choose distinct fibers
$
 F_1, \dots, F_k
$
of the morphism $\pi$. Let
$
 S \subset \bigcup _{ k } F_k 
$
be a finite subset and let
$
 f \colon \Xtilde \to X
$ be the blow-up of $X$ along $S$.
Then 
$
 - K _{ \Xtilde }
$
is big.
This follows from the bigness of $-K_X - \sum_i F_i $ and \pref{cr:big}.
Moreover, the anti-canonical divisor is still big whenever we blow up at points in the strict transform of $\bigcup _{ k } F_k $.
\end{example}

\begin{example}
Suppose $g(C)=1$ and $\deg L = 1$.
Let $D$ be an effective divisor corresponding to a global section of
$
 \cO _{ X }(1) \otimes \pi ^* L ^{ - 1 }
$
(consider the case when it exists).
Then $-K_X - D$ is big by \pref{cr:big}, so that for the blow-up $\Xtilde$ of $X$ in a finite set of points on $D$
the anti-canonical sheaf $-K_{\Xtilde}$ is always big.
\end{example}

\bibliographystyle{amsalpha}
\bibliography{bib}

\newcommand{\etalchar}[1]{$^{#1}$}
\providecommand{\bysame}{\leavevmode\hbox to3em{\hrulefill}\thinspace}
\providecommand{\MR}{\relax\ifhmode\unskip\space\fi MR }
\providecommand{\MRhref}[2]{%
  \href{http://www.ams.org/mathscinet-getitem?mr=#1}{#2}
}
\providecommand{\href}[2]{#2}
\begin{thebibliography}{TVAV11}

\bibitem[BBDG06]{MR2264108}
Lesya Bodnarchuk, Igor Burban, Yuriy Drozd, and Gert-Martin Greuel,
  \emph{Vector bundles and torsion free sheaves on degenerations of elliptic
  curves}, Global aspects of complex geometry, Springer, Berlin, 2006,
  pp.~83--128. \MR{2264108}

\bibitem[FGI{\etalchar{+}}05]{MR2222646}
Barbara Fantechi, Lothar G{\"o}ttsche, Luc Illusie, Steven~L. Kleiman, Nitin
  Nitsure, and Angelo Vistoli, \emph{Fundamental algebraic geometry},
  Mathematical Surveys and Monographs, vol. 123, American Mathematical Society,
  Providence, RI, 2005, Grothendieck's FGA explained. \MR{2222646}

\bibitem[Gie73]{MR0325616}
David Gieseker, \emph{Stable vector bundles and the {F}robenius morphism}, Ann.
  Sci. {\'E}cole Norm. Sup. (4) \textbf{6} (1973), 95--101. \MR{0325616}

\bibitem[Har77]{MR0463157}
Robin Hartshorne, \emph{Algebraic geometry}, Springer-Verlag, New
  York-Heidelberg, 1977, Graduate Texts in Mathematics, No. 52. \MR{0463157}

\bibitem[KM98]{MR1658959}
J{\'a}nos Koll{\'a}r and Shigefumi Mori, \emph{Birational geometry of algebraic
  varieties}, Cambridge Tracts in Mathematics, vol. 134, Cambridge University
  Press, Cambridge, 1998, With the collaboration of C. H. Clemens and A. Corti,
  Translated from the 1998 Japanese original. \MR{1658959}

\bibitem[Lan09]{MR2483939}
Adrian Langer, \emph{Moduli spaces of sheaves and principal {$G$}-bundles},
  Algebraic geometry---{S}eattle 2005. {P}art 1, Proc. Sympos. Pure Math.,
  vol.~80, Amer. Math. Soc., Providence, RI, 2009, pp.~273--308. \MR{2483939}

\bibitem[LP97]{MR1428426}
J.~Le~Potier, \emph{Lectures on vector bundles}, Cambridge Studies in Advanced
  Mathematics, vol.~54, Cambridge University Press, Cambridge, 1997, Translated
  by A. Maciocia. \MR{1428426}

\bibitem[Miy87]{MR946247}
Yoichi Miyaoka, \emph{The {C}hern classes and {K}odaira dimension of a minimal
  variety}, Algebraic geometry, {S}endai, 1985, Adv. Stud. Pure Math., vol.~10,
  North-Holland, Amsterdam, 1987, pp.~449--476. \MR{946247 (89k:14022)}

\bibitem[Mor98]{MR1488349}
Atsushi Moriwaki, \emph{Relative {B}ogomolov's inequality and the cone of
  positive divisors on the moduli space of stable curves}, J. Amer. Math. Soc.
  \textbf{11} (1998), no.~3, 569--600. \MR{1488349}

\bibitem[MR83]{MR714755}
V.~B. Mehta and A.~Ramanathan, \emph{Homogeneous bundles in characteristic
  {$p$}}, Algebraic geometry---open problems ({R}avello, 1982), Lecture Notes
  in Math., vol. 997, Springer, Berlin, 1983, pp.~315--320. \MR{714755}

\bibitem[Mum08]{MR2514037}
David Mumford, \emph{Abelian varieties}, Tata Institute of Fundamental Research
  Studies in Mathematics, vol.~5, Published for the Tata Institute of
  Fundamental Research, Bombay; by Hindustan Book Agency, New Delhi, 2008, With
  appendices by C. P. Ramanujam and Yuri Manin, Corrected reprint of the second
  (1974) edition. \MR{2514037}

\bibitem[Nak04]{MR2104208}
Noboru Nakayama, \emph{Zariski-decomposition and abundance}, MSJ Memoirs,
  vol.~14, Mathematical Society of Japan, Tokyo, 2004. \MR{2104208}

\bibitem[Sch01]{MR1821186}
Stefan Schr{\"o}er, \emph{Normal del {P}ezzo surfaces containing a nonrational
  singularity}, Manuscripta Math. \textbf{104} (2001), no.~2, 257--274.
  \MR{1821186}

\bibitem[Tan14]{MR3319838}
Hiromu Tanaka, \emph{Minimal models and abundance for positive characteristic
  log surfaces}, Nagoya Math. J. \textbf{216} (2014), 1--70. \MR{3319838}

\bibitem[Tan15]{MR3383599}
\bysame, \emph{The {X}-method for klt surfaces in positive characteristic}, J.
  Algebraic Geom. \textbf{24} (2015), no.~4, 605--628. \MR{3383599}

\bibitem[TVAV11]{MR2824848}
Damiano Testa, Anthony V{\'a}rilly-Alvarado, and Mauricio Velasco, \emph{Big
  rational surfaces}, Math. Ann. \textbf{351} (2011), no.~1, 95--107.
  \MR{2824848}

\end{thebibliography}

\end{document}